\documentclass[11pt]{amsart}
\usepackage{amssymb}
\usepackage{verbatim}
\usepackage{amscd}
\usepackage{amsmath}
\newtheorem{thm}{Theorem}[section]

\newcommand{\EM}{the Eilenberg--Moore spectral sequence}

\newcommand{\Tor}{\rm{Tor}}

\newcommand{\Cot}{\rm{Cotor}}

\begin{document}
\pagestyle{plain}
\title{The module category weight of compact exceptional  Lie groups}
\author{Younggi Choi}
\address{ Department of Mathematics Education, Seoul National University,
Seoul 151-748, Korea. Fax number: 82-2-889-1747}
\email{yochoi@snu.ac.kr} 
 \abstract We compute the lower bound estimate for the module
category weight of  compact exceptional  Lie groups  by
   analyzing  several
Eilenberg--Moore type spectral sequences.
\endabstract
 \maketitle
\title

\renewcommand{\theequation}{\thesection.\arabic{equation}}
\section{Introduction}
\label{sec-0}

The Lusternik - Schnirelmann category $cat (X)$ of a topological
space $X$ is the least integer $n$  such that there exists an open
cover $X=U_{1} \cup \cdots  \cup U_{n+1}$ with each $U_{i}$
contractible to a point in $X$. There are other computable
homotopy invariants  such as cup length, category weight, and
module category weight with the relation \cite{IK} : $ cup(X;\Bbb{F}_{p})
\leq wgt(X;\Bbb{F}_{p}) \leq Mwgt(X;\Bbb{F}_{p}) \leq  cat(X)$.

 Toomer introduced the explicit formula for the difference between the cup
 length
 and the category weight. Using the  formula he calculated the difference $cup(X;\Bbb{F}_{p})-wgt(X;\Bbb{F}_{p}) $
 of any simply connected compact simple Lie group \cite{T}. In fact,
 it is precisely $F_4, E_6, E_7, E_8$ which yield a positive
difference.

 On the other hands, Iwase and Kono \cite{IK} determined
 $cat (Spin(9))=8$ by  computing the lower bound of the difference between the category weight  and the module category weight of
$Spin(9)$, which  is  $Mwgt(Spin(9);\Bbb{F}_{2}) -
wgt(Spin(9);\Bbb{F}_{2} )\geq 2$.

In this paper we compute the lower bound estimate for the module
category weight of exceptional compact simple Lie groups by
studying   the difference between the category weight and the
module category weight.

This paper is organized as follows. In section 2, we collects some
known facts, which will be used in next sections. In section 3, we
compute  the module category weight  with respect to $\Bbb{F}_{2}$
coefficients of compact exceptional   Lie groups by
   analyzing  several
Eilenberg--Moore type spectral sequences. In section 4, we compute
the module category weight  with respect to $\Bbb{F}_{3}$
coefficients of compact exceptional  Lie groups by the similar
method as in the case of $\Bbb{F}_{2}$ coefficients.

We would like to thank  Professor M. Mimura and T. Nishimoto and for his
suggestions and valuable comments.

\section{Some Known Facts }
\label{sec-5} \setcounter{thm}{0} \noindent
 Throughout this paper,
the subscript of an element always means the degree of the
element, for example, the degree of $x_{i}$ is $i$. Let $E(x)$ be
the exterior algebra on $x$ and $\Bbb{F}_{2}[x]$ be the polynomial
 algebra on $x$   and $\Gamma (x)$ be the divided power algebra on $x$ which
 is free over $\gamma_{i} (x)$ with coproduct
 \[ \Delta (\gamma_{n}(x)) = \sum _{i=0}^{n} \gamma _{n-i}(x)\otimes
 \gamma _{i}(x) \]
 and the product
 \[   \gamma _{i}(x) \gamma _{j}(x)={i+j\choose i} \gamma _{i+j}(x) .\]

 We define
$cup(X;\Bbb{F}_{p})$, the  cup-length with respect to
$\Bbb{F}_{p}$, by the least integer $m$ such that  $x_{1}\cdots
x_{m+1}=0$ for any $m + 1$ elements $x_{i}\in
H^{*}(X;\Bbb{F}_{p})$.
 Let $P^{m}(\Omega X)$ be
the $m$ th projective space, in the sense of Stasheff \cite{ST},
such that there is a homotopy equivalence $P^{\infty}(\Omega
X)\simeq X$. Let  $e_{m}:P^{m}(\Omega X) \rightarrow
¡æP^{\infty}(\Omega X)\simeq X$ be the inclusion map. Consider
$(e_{m})^{*}:H^*(X;\mathbb{F}_{p}) \rightarrow H^*(P^{m}(\Omega
X);\mathbb{F}_{p})$. Then we can define category weight $
wgt(X;\Bbb{F}_{p})$ and module category weight
$Mwgt(X;\Bbb{F}_{p})$ as follows\cite{IK}:
$$\begin{array}{rcl}
 wgt(X;\Bbb{F}_{p}) &=& {\rm min} \{  m  |\,
        (e_{m})^{*}  \mbox{ is a   monomorhsim}
        \}, \\
 Mwgt(X;\Bbb{F}_{p})&=& {\rm min} \{  m  |\, (e_{m})^{*}  \mbox{ is a
split monomorphism of all Steenrod algebra module}
 \}.
\end{array} $$
\noindent Then we have the following relation \cite{IK}.
$$cup(X;\Bbb{F}_{p}) \leq wgt(X;\Bbb{F}_{p}) \leq
Mwgt(X;\Bbb{F}_{p}) \leq  cat(X).$$ Now we list the mod $p$
cohomology of the exceptional Lie groups. We refer \cite{MT} for
the condensed treatment of these cohomology including Hopf algebra
structure and the action of the Steenrod algebra.

\begin{thm}
\label{mod2} The mod 2 cohomology of the exceptional Lie groups
$G_2$, $F_4$, $E_{6}$, $E_{7}$, and $E_{8}$ are as follows.
\begin{eqnarray*}
H^*(G_2;\mathbb{F}_{2}) &\cong & \mathbb{F}_{2}[x_3]/(x_3^4) \otimes E(Sq^2 x_3)\\
H^*(F_4;\mathbb{F}_{2})&\cong & \mathbb{F}_{2}[x_3]/(x_3^4)
\otimes E(Sq^2 x_3,x_{15},Sq^8x_{15}) \\
 H^*(E_{6};\mathbb{F}_{2}) &\cong &
      \mathbb{F}_{2}[x_3]/(x_3^4)\otimes
          E(Sq^2x_3,Sq^{4,2}x_3,x_{15},Sq^{8,4,2}x_3,Sq^8x_{15})\\
 H^*(E_{7};\mathbb{F}_{2})&\cong &
\mathbb{F}_{2}[x_3,Sq^2x_3,Sq^{4,2}x_3]/(x_3^4, (Sq^2x_3)^4, (Sq^{4,2}x_3)^4)\\
                       &   &\otimes E(x_{15},Sq^{8,4,2}x_3,Sq^8x_{15},Sq^{4,8}x_{15})\\
 H^*(E_{8};\mathbb{F}_{2})&\cong &
\mathbb{F}_{2}[x_3]/(x_3^{16})\otimes
\mathbb{F}_{2}[Sq^2x_3]/((Sq^2 x_3)^8)\\
         &  & \otimes \mathbb{F}_{2}[Sq^{4,2}x_3,x_{15}]/((Sq^{4,2}x_3)^4, x_{15}^4)\\
         &  & \otimes E(Sq^{8,4,2}x_3,Sq^8x_{15},Sq^{4,8}x_{15},Sq^{2,4,8}x_{15}).\\
\end{eqnarray*}
\end{thm}

\begin{thm}
\label{mod3}The mod 3 cohomology of the exceptional Lie groups
$G_2$, $F_4$, $E_{6}$, $E_{7}$, and $E_{8}$ are as follows.
\begin{eqnarray*}
H^*(G_2;\mathbb{F}_{3}) &\cong & E(x_3, x_{11})\\
H^*(F_4;\mathbb{F}_{3})&\cong & \mathbb{F}_{3}[\beta {\mathcal
P}^1 x_3]/((\beta {\mathcal P}^1 x_3)^3)
        \otimes  E(x_3,{\mathcal P}^1x_3, x_{11},{\mathcal P}^1x_{11}) \\
 H^*(E_6;\mathbb{F}_{3}) & \cong  &
     \mathbb{F}_{3}[\beta {\mathcal P}^1
x_3]/((\beta {\mathcal P}^1 x_3)^3)
        \otimes
 E(x_3,{\mathcal P}^1x_3,x_9,x_{11},{\mathcal P}^1x_{11},x_{17})\\
H^*(E_7;\mathbb{F}_{3})& \cong & \mathbb{F}_{3}[\beta {\mathcal
P}^1x_3]/((\beta
{\mathcal P}^1x_3)^3)\\
                   &  &  \otimes   E(x_3,{\mathcal P}^1x_3,x_{11},{\mathcal P}^1x_{11},
                             {\mathcal P}^{3,1}x_3,x_{27},x_{35})\\
H^*(E_8;\mathbb{F}_{3}) & \cong & \mathbb{F}_{3}[\beta {\mathcal
                P}^1x_3,\beta {\mathcal P}^{3,1}x_3]/((\beta {\mathcal P}^1x_3)^3,(\beta
                    {\mathcal P}^{3,1}x_3)^3)\\
   &  & \otimes E(x_3,{\mathcal P}^1x_3,x_{15},{\mathcal P}^{3,1}x_3,{\mathcal P}^3x_{15},
      x_{35},x_{39},x_{47}).
\end{eqnarray*}
\end{thm}

\section{ Module Category Weight  with respect to $\Bbb{F}_{2}$ coefficients }  \label{sec-4}
\setcounter{thm}{0}
 \setcounter{equation}{0}

 Let  $\tilde{G}$ be  the 3-connected cover of $G$ which is the
homotopy fibre of the map $G\buildrel\iota\over\longrightarrow
K(Z,3)$ where $\iota$ is the fundamental class of  $H^{3}(G;Z)$.
Then we have the following   fibrations:  $CP^{\infty }\to
\tilde{G} \to G $,  $S^{1}\to \Omega \tilde{G} \to \Omega G $.

\begin{thm}{\rm \cite{KM,M}}
\label{conn1}
 The mod 2 cohomology of the 3-connected covers of the
exceptional Lie groups $\tilde{G}_2$, $\tilde{F}_4$,
$\tilde{E}_{6}$, $\tilde{E}_{7}$, and $\tilde{E}_{8}$ are as
follows.

\begin{eqnarray*}
H^*(\tilde{G}_2;\mathbb{F}_{2}) &\cong & \mathbb{F}_{2}[x_8]\otimes E(Sq^1x_8, Sq^{2,1}x_8)\\
H^*(\tilde{F}_4;\mathbb{F}_{2})&\cong & \mathbb{F}_{2}[x_8]\otimes E(Sq^1x_8,Sq^{2,1}x_8,Sq^{4,2,1}x_8,Sq^{8,4,2,1}x_8) \\
 H^*(\tilde{E}_{6};\mathbb{F}_{2}) &\cong & \mathbb{F}_{2}[x_{32}]\otimes
              E(x_9, Sq^{2} x_9,Sq^{4,2} x_9,Sq^{8} x_9,x_{23},Sq^{16,8} x_9)\\
 H^*(\tilde{E}_{7};\mathbb{F}_{2})&\cong &
\mathbb{F}_{2}[x_{32}]\otimes
E(x_{11},Sq^{4}x_{11},Sq^{8}x_{11},x_{23},Sq^{8,8}x_{11},Sq^1x_{32},Sq^{16,8}x_{11})
              \\
 H^*(\tilde{E}_{8};\mathbb{F}_{2})&\cong &
 \mathbb{F}_{2}[x_{15}]/(x_{15}^{4})\otimes \mathbb{F}_{2}[x_{32}] \\
 && \otimes
               E(x_{23},x_{27},x_{29},Sq^1x_{32},x_{35},Sq^{4}x_{35},Sq^{8,4}x_{35}).\\
\end{eqnarray*}

 \end{thm}

To get the module category weight  of   exceptional Lie groups
$G$, we
 study
  {\EM} converging to
$H^{*}(G )$ with $ E_{2} \cong \Cot_{H^{*}(\Omega G ;\Bbb{F}_{2}
)}(\Bbb{F}_{2}, \Bbb{F}_{2})$.  This is a spectral sequence of
Hopf algebras but it depends on the coalgebra structure. So we
should determine the coalgebra structure of $H^{*}(\Omega
G;\Bbb{F}_{2})$. To get the coalgebra structure of $H^{*}(\Omega
G;\Bbb{F}_{2})$, we consider {\EM} converging to
 $H^{*}(\Omega G;\Bbb{F}_{2})$
with $$\begin{array}{rcl}
 E_{2}& \cong&
 {\Tor}_{H^{*}(G;\Bbb{F}_{2})}(\Bbb{F}_{2},\Bbb{F}_{2})
   \end{array} $$
\noindent  Since  $E_{2}$ concentrates in the even dimensions, the
spectral sequence
 collapses at the $E_{2}$-term, i.e., $E_{2}=E_{\infty }$. Then there is no coalgebra extension problem in such a spectral sequence \cite{K}.
 We refer the reader to \cite{Mc} for concise treatment of  above Eilenberg Moore spectral sequence. So as
 a coalgebra we have the following.

\begin{thm} The coalgebra structure of the mod 2 cohomology of the loop spaces
of  exceptional Lie groups $G_2$, $F_4$, $E_{6}$, $E_{7}$, and
$E_{8}$ are as follows.
\begin{eqnarray*}
H^*(\Omega G_2;\mathbb{F}_{2})&\cong & E(a_2) \otimes
                    \Gamma (a_4, b_{10})\\
                 H^*(\Omega F_4;\mathbb{F}_{2}) &\cong &
    E(a_2) \otimes \Gamma (a_4, b_{10}, a_{14}, a_{16}, a_{22})\\
 H^*(\Omega
E_6;\mathbb{F}_{2})&\cong &
 E(a_2) \otimes \Gamma (a_4, a_{8}, b_{10},  a_{14}, a_{16}, a_{22})\\
 H^*(\Omega E_7;\mathbb{F}_{2})&\cong &
     E(a_2, a_{4}, a_{8}) \otimes \Gamma (  b_{10},  a_{14}, a_{16}, b_{18}, a_{22}, a_{26},  b_{34}  )\\
 H^*(\Omega E_8;\mathbb{F}_{2})&\cong &
         E(a_2, a_{4}, a_{8}, a_{14} ) \otimes \Gamma (   a_{16},    a_{22}, a_{26}, a_{28},  b_{34}, b_{38}, b_{46},  b_{58}  )\\
        \end{eqnarray*}
\noindent especially we have  $Sq^{4}b_{10} =a_{14}$ and
$Sq^{8}b_{18} =a_{26}$ by Theorem \ref{conn1}.
 \end{thm}

Now we consider {\EM}   converging to $H^{*}(G;\Bbb{F}_{2})$ with

\setcounter{equation}{\value{thm}}
\begin{equation}
\label{cotor} E_{2} \cong  \Cot_{H^{*}(\Omega
G;\Bbb{F}_{2}})(\Bbb{F}_{2}, \Bbb{F}_{2}).
\end{equation}
\setcounter{thm}{\value{equation}}

\noindent Then we get the following theorem by   the formal Cotor
computation. We refer the reader to \cite{MS} for detail
computation method  of this spectral sequence.

\begin{thm}
\label{thmcot}
 $\Cot_{H^{*}(\Omega G;\Bbb{F}_{2})}(\Bbb{F}_{2},
\Bbb{F}_{2})$  of the exceptional Lie groups G  for $G_2$, $F_4$,
$E_{6}$, $E_{7}$, and $E_{8}$ are as follows.
\begin{eqnarray*}
\Cot_{H^{*}(\Omega G_2;\Bbb{F}_{2})}(\Bbb{F}_{2},
\Bbb{F}_{2}) &\cong & \mathbb{F}_{2}[x_3] \otimes E( x_5, z_{11})\\
\Cot_{H^{*}(\Omega F_4;\Bbb{F}_{2})}(\Bbb{F}_{2}, \Bbb{F}_{2})
&\cong & \mathbb{F}_{2}[x_3]
\otimes E( x_5, z_{11}, x_{15}, x_{23}) \\
\Cot_{H^{*}(\Omega E_6;\Bbb{F}_{2})}(\Bbb{F}_{2}, \Bbb{F}_{2})
&\cong &
      \mathbb{F}_{2}[x_3]\otimes
          E(x_5, x_9, z_{11}, x_{15}, x_{17} ,x_{23})\\
 \Cot_{H^{*}(\Omega E_7;\Bbb{F}_{2})}(\Bbb{F}_{2}, \Bbb{F}_{2}) &\cong &
\mathbb{F}_{2}[x_3, x_5,  x_{9}]
                       \otimes E(z_{11}, x_{15}, x_{17}, z_{19}, x_{23},  x_{27}, z_{35})\\
 \Cot_{H^{*}(\Omega E_8;\Bbb{F}_{2})}(\Bbb{F}_{2}, \Bbb{F}_{2}) &\cong &
\mathbb{F}_{2}[x_3, x_{5}, x_9 ,x_{15}]\otimes E( x_{17},   x_{23}, x_{27}, x_{29},  z_{35}, z_{39}, z_{47},z_{59}    )\\
\end{eqnarray*}
especially we have  $Sq^{4}z_{11} =x_{15}$ and $Sq^{8}z_{19}
=x_{27}$.
\end{thm}

\noindent Then from information Theorem \ref{mod2} of
$H^{*}(G;\Bbb{F}_{2})$, we can analyze non trivial differentials
of  {\EM}  ( \ref{cotor} ) converging to $H^{*}(G;\Bbb{F}_{2})$
as follows:

\setcounter{equation}{\value{thm}}
\begin{equation}
\label{diff0}
\begin{array}{rcl}
d_{3}(  z_{11}  )& = & x_{3}^{4}\quad \mbox{ for } G=  G_2, F_4, E_{6}, E_{7}\\
d_{3}(  z_{19}  )& = & x_{5}^{4}\quad \mbox{ for } G=  E_{7}\\
d_{3}(  z_{35}  )& = & x_{9}^{4}\quad \mbox{ for } G=   E_{7}, E_{8} \\
d_{7}(  z_{39}  )& = & x_{5}^{8}\quad \mbox{ for } G=  E_{8} \\
d_{15}(  z_{47}  )& = & x_{3}^{16}\quad \mbox{for } G=  E_{8} \\
d_{3}(  z_{59}  )& = & x_{15}^{4}\quad  \mbox{for } G=  E_{8}.\\
         \end{array}
\end{equation}
\setcounter{thm}{\value{equation}}

 Next as
in \cite{I,IK}, truncating the above computation with the same
differential $d_{i}$ in (\ref{diff0} ), we can compute   the
spectral sequence of Stasheff's type converging to $H^{*}(P^{m} (
\Omega G);\Bbb{F}_{2})$.
 Let $A= H^{*}(G;\Bbb{F}_{2})$
in Theorem \ref{mod2}. Then like the result in \cite [ Proposition
2.1]{IK}, for low $m$ such as $1 \leq m \leq 3$, we have the
following:

\setcounter{equation}{\value{thm}}
\begin{equation}
\label{stas} H^{*}(P^{m}(\Omega
   G);\Bbb{F}_{2})=A^{[m]}\oplus \sum_{i} z_{4i+3}\cdot A^{[m-1]}\oplus
   S_{m},\left\{ \begin{array}{ll}
   i=3,  \mbox { for } G= G_2, F_4, E_{6} \\
   i=3,4,8,  \mbox { for } G= E_{7} \\
     i=8,9,11,14, \mbox { for }  G= E_{8}
  \end{array}
  \right.
   \end{equation}
\setcounter{thm}{\value{equation}}

\noindent as modules where $A^{[m]}, (m\geq 0)$ denotes the
quotient module $A/D^{m+1}(A)$ of $A$ by the submodule $D^{m+1}(A)
\subseteq A$ generated by all the products of $m+1$ elements in
positive dimensions in $A$, $z_{4i+3}\cdot A^{[m-1]}$ denotes a
submodule corresponding to a submodule in $A\otimes E(z_{4i+3})$
and $S_{m}$ satisfies $ S_{m}\cdot {\tilde H}^{*}(P^{m}(\Omega
G);\Bbb{F}_{2}) = 0$ and $S_{m}|_{P^{m-1}(\Omega G)} = 0$. For
more detail for $S_{m}$, we refer the paper \cite{I}. Now we
compute the module category weight using the similar method in
\cite{C,IK}.

\begin{thm}
\label{thm3-6} The module category weight is as follows:
\begin{eqnarray*}
Mwgt( G_2 ;\Bbb{F}_{2}) & \geq  & 4, \\
Mwgt( F_4 ;\Bbb{F}_{2}) & \geq  & 8, \\
Mwgt( E_6 ;\Bbb{F}_{2}) & \geq  & 10, \\
Mwgt( E_7 ;\Bbb{F}_{2}) & \geq  & 15, \\
Mwgt( E_8 ;\Bbb{F}_{2}) & \geq  & 32 . \\
        \end{eqnarray*}
        \setcounter{thm}{\value{equation}}
   \end{thm}
   \begin{proof}

 From Theorem \ref{thmcot},
 $Sq^{4}z_{11} =x_{15}$
    in $H^{*}(P^{1}(\Omega
   G);\Bbb{F}_{2})$ for $G=F_4, E_6, E_7$.
Then from (\ref{stas}),  $Sq^{4}z_{11} =x_{15}$ modulo $S_{2}$
     in $H^{*}(P^{2}(\Omega
   G);\Bbb{F}_{2})$ for $G=F_4, E_6, E_7$. Since  $S_{2}$ is
   even-dimensional \cite{C,IK}, the  modulo $S_{2}$ is trivial so  $Sq^{4}z_{11} =x_{15}$
      in $H^{*}(P^{2}(\Omega G);\Bbb{F}_{2})$.   Therefore  in $H^{*}
(P^{7}(\Omega
  F_4 );\Bbb{F}_{2})$, $H^{*}
(P^{9}(\Omega
  E_{6});\Bbb{F}_{2})$, and  $H^{*}(P^{14}(\Omega E_{7});\Bbb{F}_{2}) $, we have
\begin{eqnarray*} \label{module}
 Sq^{4}( x_{3}^{3}x_{5}z_{11}x_{23})& = &  x_{3}^{3}x_{5}x_{15}x_{23} \mbox{ for }  F_4, \\
 Sq^{4}( x_{3}^{3}x_{5}x_{9}z_{11}x_{17}x_{23}) & = & x_{3}^{3}x_{5}x_{9}x_{15}x_{17}x_{23} \mbox{ for }  E_6, \\
 Sq^{4}( x_{3}^{3}x_{5}^{3}x_{9}^{3}z_{11}x_{17}x_{23}x_{27}
   ) & = & x_{3}^{3}x_{5}^{3}x_{9}^{3}x_{15}x_{17}x_{23}x_{27} \mbox{ for }  E_7. \\
        \end{eqnarray*}
\noindent By the definition in section 2, $Mwgt(X;\Bbb{F}_{2})$ is
the least $m$ such that $(e_{m})^{*}$
 is a split monomorphism of all Steenrod algebra module.  Let  $\phi_{m}: H^{*}(P^{m}( \Omega G);\Bbb{F}_{2})
\rightarrow H^{*}(G;\Bbb{F}_{2})$ be a epimorphism which preserves
all Steenrod actions and $\phi_{m} \circ (e_{m})^{*}\cong 1_{
H^{*}(G;\Bbb{F}_{2})}$. Suppose that there are epimorphisms
\begin{eqnarray*} \label{map}
 \phi_{7}:H^{*}
(P^{7}(\Omega
  F_4 );\Bbb{F}_{2}) & \rightarrow &  H^{*}(F_4 ;\Bbb{F}_{2}), \\
  \phi_{9}:H^{*}
(P^{9}(\Omega
  E_{6});\Bbb{F}_{2}) & \rightarrow &  H^{*}(E_6 ;\Bbb{F}_{2}), \\
 \phi_{14}:H^{*}
(P^{14}(\Omega
  E_{7});\Bbb{F}_{2}) & \rightarrow &  H^{*}(E_7 ;\Bbb{F}_{2}). \\
        \end{eqnarray*}

   \noindent  Then we have  the following
   diagrams:
   \setcounter{equation}{\value{thm}}
\begin{equation}
\label{modu}
\begin{array}{ccccc}
H^{*} (P^{7}(\Omega
  F_4 );\Bbb{F}_{2})  &\stackrel{\phi_{7}  }{\longrightarrow}
&   H^{*}(F_4 ;\Bbb{F}_{2})  \\
   x_{3}^{3}x_{5}x_{15}x_{23}  &\stackrel{}{\longmapsto}
&  x_{3}^{3}x_{5}x_{15}x_{23}  \\
 \stackrel{}{ Sq^{4} }{\uparrow } & &\stackrel{}{ Sq^{4}} \uparrow  \\
x_{3}^{3}x_{5}z_{11}x_{23}  &\stackrel{}{\longmapsto}
& \qquad  0 \\
& & \\
 H^{*} (P^{9}(\Omega
  E_{6});\Bbb{F}_{2})  &\stackrel{\phi_{9}  }{\longrightarrow}
&  H^{*}(E_6 ;\Bbb{F}_{2})  \\
   x_{3}^{3}x_{5}x_{9}x_{15}x_{17}x_{23} &\stackrel{}{\longmapsto}
&  x_{3}^{3}x_{5}x_{9}x_{15}x_{17}x_{23}  \\
\stackrel{}{ Sq^{4}}{\uparrow } & &\stackrel{}{ Sq^{4}} \uparrow  \\
 x_{3}^{3}x_{5}x_{9}z_{11}x_{17}x_{23} &\stackrel{}{\longmapsto}
& \qquad   0 \\
& & \\
H^{*} (P^{14}(\Omega
  E_{7});\Bbb{F}_{2})  &\stackrel{\phi_{14}  }{\longrightarrow}
&  H^{*}(E_7 ;\Bbb{F}_{2})  \\
   x_{3}^{3}x_{5}^{3}x_{9}^{3}x_{15}x_{17}x_{23}x_{27}  &\stackrel{}{\longmapsto}
&  x_{3}^{3}x_{5}^{3}x_{9}^{3}x_{15}x_{17}x_{23}x_{27}  \\
\stackrel{}{ Sq^{4}}{\uparrow } & &\stackrel{}{ Sq^{4}} \uparrow  \\
 x_{3}^{3}x_{5}^{3}x_{9}^{3}z_{11}x_{17}x_{23}x_{27} &\stackrel{}{\longmapsto}
& \qquad   0 \\
\end{array}
\end{equation}
\setcounter{thm}{\value{equation}}

\noindent Obviously this is a contradiction. So  $\phi_{7}$,
$\phi_{9}$, and $\phi_{14}$ are not  epimorphisms. This means that
$(e_{7})^{*}$, $(e_{9})^{*}$, and $(e_{14})^{*}$
 can not be   split monomorphisms of all Steenrod algebra module. Hence we obtain that

 $$Mwgt( F_4   ; \Bbb{F}_{2} ) \geq  8,  \,\,\,  Mwgt( E_{6}   ; \Bbb{F}_{2} ) \geq  10,  \,\,\,    Mwgt( E_{7}   ; \Bbb{F}_{2} ) \geq  15 \, .$$

\noindent Now we consider the category weight. For $G_{2}$,
$x_{3}^{3}x_{5}\in H^{*} (P^{4}(\Omega
  G_2 );\Bbb{F}_{2} )$. Hence $(e_{4})^{*}$ is a monomorphism, so  $ wgt( G_{4} ;\Bbb{F}_{2})= 4 $.
  By the same way,  $ wgt( G ;\Bbb{F}_{2})$ is 6 for $G=F_4 $, 8 for $G=E_6 $, 13 for $G=E_7 $, 32
for $G=E_8 $.   In fact  the category weight is the same as the
Toomer's invariant, the filtration length,  that is,
$wgt(G;\Bbb{F}_{2})=f_{2}(G)$ in \cite{T}.

 For the case of $G_{2}$ and $E_{8}$, by dimensional
reason,
  any generator of type  $x_{-}$ can not be of the form
   $Sq^{i}( z_{-} )$ for any $i$ and for any generator of type
   $z$. So  we can not apply the method in (\ref{modu}). Hence  we do not obtain  any positive difference between the category
   weight and the module category weight.
Hence  we have
$$ Mwgt( G_2 ;\Bbb{F}_{2})\geq wgt( G_2 ;\Bbb{F}_{2}) = 4, \, \, \, Mwgt( E_8 ;\Bbb{F}_{2})\geq wgt( E_8 ;\Bbb{F}_{2}) = 32  \, .$$
\end{proof}
\noindent Summarizing above results, we have the following
results.

\bigskip
 \centerline{
 \vbox{\offinterlineskip
\tabskip=0pt \halign {
   \strut   \vrule#&\quad # \quad &  \vrule#&\quad # \quad &
    \vrule#& \quad  # \qquad &  \vrule#& \quad # &  \vrule#& #\cr
\noalign{\hrule} & X &&  $ wgt(X;\Bbb{F}_{2})$&&
$Mwgt(X;\Bbb{F}_{2}) $ & & $cat(X) $ & \cr
 \noalign{\hrule} &  $G_2$ &&
 4
&& $\geq $  4 &&4 &\cr \noalign{\hrule} & $F_4$ &&
 6
&& $\geq $ 8 &&?&\cr \noalign{\hrule} & $E_{6}$ &&
 8
&& $\geq $ 10 && ?&\cr
 \noalign{\hrule} & $E_{7}$ &&
 13
&& $ \geq $ 15 &&?&\cr \noalign{\hrule} &  $E_{8}$ &&
 32
&& $\geq $ 32 && ? &\cr
 \noalign{\hrule} }   }  }

\smallskip

\section{ Module Category Weight   with respect to $\Bbb{F}_{3}$ coefficients }  \label{sec-4}
\setcounter{thm}{0}
 \setcounter{equation}{0}

Now we turn to  the case of  $\Bbb{F}_{3}$ coefficients.

\begin{thm}{\rm \cite{ KM,M}}
\label{conn2} The mod 3 cohomology of the 3-connected covers of
the exceptional Lie groups $\tilde{G}_2$, $\tilde{F}_4$,
$\tilde{E}_{6}$, $\tilde{E}_{7}$, and $\tilde{E}_{8}$ are as
follows.

\begin{eqnarray*}
H^*(\tilde{G}_2;\mathbb{F}_{3}) &\cong & \mathbb{F}_{3}[y_6]\otimes E(x_{11}, \beta  y_6)\\
H^*(\tilde{F}_4;\mathbb{F}_{3})&\cong & \mathbb{F}_{3}[y_{18}]\otimes E( x_{11},{\mathcal P}^1x_{11},\beta y_{18}, {\mathcal P}^1 \beta y_{18} )  \\
 H^*(\tilde{E}_{6};\mathbb{F}_{3}) &\cong & \mathbb{F}_{3}[y_{18}]\otimes
      E(x_9,x_{11},{\mathcal P}^1x_{11},x_{17}, \beta y_{18}, {\mathcal P}^1 \beta y_{18}  )        \\
 H^*(\tilde{E}_{7};\mathbb{F}_{3})&\cong &
\mathbb{F}_{3}[y_{54}]\otimes
      E(x_{11},{\mathcal P}^1x_{11},x_{19}, {\mathcal P}^1  x_{19}, {\mathcal P}^2 x_{19}, \beta y_{54} ) \\
 H^*(\tilde{E}_{8};\mathbb{F}_{3})&\cong &
\mathbb{F}_{3}[y_{54}]\otimes
      E(x_{15},z_{23}, {\mathcal P}^{1}z_{23},
      x_{35},x_{39},x_{47},  \beta y_{54}, y_{59} ).  \\
\end{eqnarray*}
 \end{thm}

\noindent Following  the similar method as in the case of
$\Bbb{F}_{2}$ coefficients, we have the next theorem.

 \begin{thm}
 \label{comod3}The coalgebra structure of the mod 3 cohomology of the loop spaces
of  exceptional Lie groups $G_2$, $F_4$, $E_{6}$, $E_{7}$, and
$E_{8}$ are as follows.
\begin{eqnarray*}
H^*(\Omega G_2;\mathbb{F}_{3})&\cong &
                    \Gamma (a_2, a_{10})\\
                 H^*(\Omega F_4;\mathbb{F}_{3}) &\cong &
   \mathbb{F}_{3}[a_2]/(a_2^3) \otimes \Gamma (a_{6}, a_{10}, a_{14},  b_{22})\\
 H^*(\Omega
E_6;\mathbb{F}_{3})&\cong &
 \mathbb{F}_{3}[a_2]/(a_2^3)  \otimes \Gamma (a_{6}, a_{8}, a_{10}, a_{14}, a_{16}, b_{22})\\
 H^*(\Omega E_7;\mathbb{F}_{3})&\cong &
     \mathbb{F}_{3}[a_2]/(a_2^3) \otimes \Gamma ( a_6,  a_{10},  a_{14}, a_{18}, b_{22}, a_{26},  a_{34}  )\\
 H^*(\Omega E_8;\mathbb{F}_{3})&\cong &
     \mathbb{F}_{3}[a_2]/(a_2^3) \otimes  \mathbb{F}_{3}[a_6]/(a_6^3)   \otimes \Gamma (  a_{14}, a_{18}, b_{22},  a_{26},  a_{34}, a_{38}, a_{46}, b_{58}  )\\
        \end{eqnarray*}
\noindent especially we have  ${\mathcal P}^1b_{22} =a_{26}$ by
Theorem  \ref{conn2} and ${\mathcal P}^1{\mathcal P}^1=2{\mathcal
P}^2$ and by  change of generators.
 \end{thm}
Consider {\EM}   converging to $H^{*}(G;\Bbb{F}_{3})$ with
\setcounter{equation}{\value{thm}}
\begin{equation}
\label{cot3}
  \begin{array}{rcl}
E_{2} &\cong & \Cot_{H^{*}(\Omega G;\Bbb{F}_{3})}(\Bbb{F}_{3},
\Bbb{F}_{3}).
\end{array}
\end{equation}
\setcounter{thm}{\value{equation}}

\begin{thm}
\label{cotmod3} $\Cot_{H^{*}(\Omega G;\Bbb{F}_{3})}(\Bbb{F}_{3},
\Bbb{F}_{3})$ of the exceptional Lie groups G  for $G_2$, $F_4$,
$E_{6}$, $E_{7}$, and $E_{8}$ are as follows.
\begin{eqnarray*}
\Cot_{H^{*}(\Omega G_2;\Bbb{F}_{3})}(\Bbb{F}_{3}, \Bbb{F}_{3})
&\cong & E(x_3, x_{11})
\\
\Cot_{H^{*}(\Omega F_4;\Bbb{F}_{3})}(\Bbb{F}_{3}, \Bbb{F}_{3})
&\cong & E(x_3)\otimes \mathbb{F}_{3}[\beta {\mathcal P}^1 x_3]
\otimes E({\mathcal P}^1x_3, x_{11},{\mathcal P}^1x_{11}, z_{23})
\\
\Cot_{H^{*}(\Omega E_6;\Bbb{F}_{3})}(\Bbb{F}_{3}, \Bbb{F}_{3})
&\cong &
      E(x_3)\otimes
\mathbb{F}_{3}[\beta {\mathcal P}^1 x_3] \otimes  E({\mathcal P}^1 x_3, x_9, x_{11},{\mathcal P}^1x_{11},x_{17}, z_{23})\\
\Cot_{H^{*}(\Omega E_7;\Bbb{F}_{3})}(\Bbb{F}_{3}, \Bbb{F}_{3})
&\cong &
      E(x_3)\otimes
\mathbb{F}_{3}[\beta {\mathcal P}^1 x_3] \otimes  E({\mathcal P}^1 x_3, x_{11},{\mathcal P}^1x_{11}, x_{19}, z_{23}, x_{27}, x_{35}) \\
 \Cot_{H^{*}(\Omega E_8;\Bbb{F}_{3})}(\Bbb{F}_{3}, \Bbb{F}_{3})  &\cong &
      E(x_3)\otimes
\mathbb{F}_{3}[\beta {\mathcal P}^1 x_3] \otimes   E({\mathcal
P}^1 x_3)\otimes \mathbb{F}_{3}[\beta {\mathcal P}^3{\mathcal P}^1
x_3] \\  && \otimes E(x_{15}, x_{19}, z_{23}, x_{27}, x_{35},
x_{39}, x_{47},z_{59} )
\end{eqnarray*}
\noindent especially we have  ${\mathcal P}^1 z_{23} =x_{27}$.
\end{thm}

\noindent Then from information Theorem \ref{mod3} of
$H^{*}(G;\Bbb{F}_{3})$, we can analyze non trivial differentials
of  {\EM} (\ref{cot3}) converging to $H^{*}(G;\Bbb{F}_{3})$
 as follows:

\setcounter{equation}{\value{thm}}
\begin{equation}
\label{diff}
\begin{array}{rcl}
d_{3}(  z_{23}  )& = & (\beta {\mathcal P}^1 x_{3})^{3}, \mbox{ for } G=  F_4, E_{6}, E_{7}\\
d_{3}(  z_{59}  )& = & (\beta {\mathcal P}^3{\mathcal P}^1
x_3)^{3},  \mbox{ for } G=  E_{8}\\
         \end{array}
\end{equation}
\setcounter{thm}{\value{equation}}

 Let $A= H^{*}(G;\Bbb{F}_{3})$
in Theorem \ref{mod3}. Then like the result in (\ref{stas}), for
low $m$ such as  $1 \leq m \leq 3$,  we have the following:

\setcounter{equation}{\value{thm}}
\begin{equation}
\label{stas2} H^{*}(P^{m}(\Omega
   G);\Bbb{F}_{3})=A^{[m]}\oplus \sum_{i} z_{4i+3}\cdot A^{[m-1]}\oplus
   S_{m}\left\{ \begin{array}{ll}
   i=5,  \mbox { for } G= F_4, E_{6}, E_{7} \\
     i=5,14, \mbox { for }  G= E_{8}
  \end{array}
  \right.
   \end{equation}
\setcounter{thm}{\value{equation}}

\noindent  as modules.  Now we compute the module category weight
using the same method in Theorem \ref{thm3-6}.

\begin{thm}
\label{mcwmod3} The module category weight is as follows:

\begin{eqnarray*}
\label{diff}
Mwgt( G_2 ;\Bbb{F}_{3}) & \geq  &   2,\\
Mwgt( F_4 ;\Bbb{F}_{3}) & \geq  & 8, \\
Mwgt( E_6 ;\Bbb{F}_{3}) & \geq  & 10 . \\
Mwgt( E_7 ;\Bbb{F}_{3}) & \geq  & 13, \\
Mwgt( E_8 ;\Bbb{F}_{3}) & \geq  & 18 . \\
        \end{eqnarray*}
   \end{thm}
   \begin{proof}
 From Theorem \ref{comod3},   we get
 ${\mathcal P}^1 z_{23} =x_{27}$
    in $H^{*}(P^{1}(\Omega
   G) ;\Bbb{F}_{3})$ for $G= E_7, E_8$.
Then ${\mathcal P}^1z_{23} =x_{27}$ modulo $S_{2}$
    in $H^{*}(P^{2}(\Omega
   G) ;\Bbb{F}_{3})$ for $G=E_7, E_8$ from (\ref{stas2}). Since $S_{2}$ is
   even-dimensional \cite{C,IK},   the  modulo $S_{2}$ is trivial and  ${\mathcal P}^{1}z_{23} =x_{27}$
      in $H^{*}(P^{2}(\Omega G) ;\Bbb{F}_{3})$.   So  in $H^{*}(P^{12}(\Omega E_{7}) ;\Bbb{F}_{3}) $ and  $H^{*}(P^{17}(\Omega E_{8}) ;\Bbb{F}_{3}) $, we have
 {\small \begin{eqnarray*}
\label{module} {\mathcal P}^1( (\beta {\mathcal P}^1 x_{3})^{2}
x_{3} x_{7} x_{11}x_{15}x_{19}z_{23}x_{35}  )&  = & (\beta
{\mathcal P}^1 x_{3})^{2} x_{3}x_{7}
x_{11}x_{15}x_{19}x_{27}x_{35}, \\
 {\mathcal P}^1( (\beta {\mathcal P}^1 x_{3})^{2}(\beta {\mathcal
P}^3{\mathcal P}^1 x_{3})^{2} x_{3} x_{7}x_{15}x_{19}z_{23} x_{35}
x_{39} x_{47} )& =& (\beta {\mathcal P}^1 x_{3})^{2}(\beta
{\mathcal P}^3{\mathcal P}^1 x_{3})^{2} x_{3} x_{7}x_{15}x_{19}
x_{27}, x_{35} x_{39} x_{47}.
  \\
        \end{eqnarray*} }
        \noindent Note that the filtration lengths of $\beta {\mathcal P}^1 x_{3}$ and $ \beta {\mathcal P}^3{\mathcal P}^1 x_{3}$ are both 2 by the result in \cite{T}.  Let  $\phi_{m}: H^{*}(P^{m}(
\Omega G);\Bbb{F}_{p}) \rightarrow H^{*}(G;\Bbb{F}_{p})$ be a
epimorphism which preserves all Steenrod actions and $\phi_{m}
\circ (e_{m})^{*}\cong 1_{ H^{*}(G;\Bbb{F}_{p})}$. Suppose that
there are epimorphisms

\begin{eqnarray*} \label{map}
  \phi_{12}:H^{*}
(P^{12}(\Omega
  E_{7}) ;\Bbb{F}_{3}) & \rightarrow &  H^{*}(E_7  ;\Bbb{F}_{3}), \\
 \phi_{17}:H^{*}
(P^{17}(\Omega
  E_{8}) ;\Bbb{F}_{3}) & \rightarrow &  H^{*}(E_8  ;\Bbb{F}_{3}). \\
        \end{eqnarray*}
 \noindent Then we have  the following
   diagrams:

    \small{    \setcounter{equation}{\value{thm}}
\begin{equation}
\label{modu3}
 \begin{array}{ccccc}
 H^{*} (P^{12}(\Omega
  E_{7}) ;\Bbb{F}_{3})  &\stackrel{\phi_{12}  }{\longrightarrow}
&  H^{*}(E_7  ;\Bbb{F}_{3})  \\
   (\beta {\mathcal P}^1 x_{3})^{2} x_{3} x_{7} x_{11}x_{15}x_{19}x_{27}x_{35} &\stackrel{}{\longmapsto}
& (\beta {\mathcal P}^1 x_{3})^{2} x_{3} x_{7} x_{11}x_{15}x_{19}x_{27}x_{35}  \\
\stackrel{}{{\mathcal P}^1}{\uparrow } & &\stackrel{}{{\mathcal P}^1} \uparrow  \\
(\beta {\mathcal P}^1 x_{3})^{2} x_{3} x_{7}
x_{11}x_{15}x_{19}z_{23}x_{35} &\stackrel{}{\longmapsto}
& \quad 0 \\
& & \\
\end{array}
\end{equation}
\setcounter{thm}{\value{equation}}

 \[      \begin{array}{ccccc}
 H^{*} (P^{17}(\Omega
  E_{78}) ;\Bbb{F}_{3})  &\stackrel{\phi_{17}  }{\longrightarrow}
&  H^{*}(E_8  ;\Bbb{F}_{3})  \\
 (\beta {\mathcal P}^1 x_{3})^{2}(\beta {\mathcal P}^3{\mathcal P}^1 x_{3})^{2} x_{3} x_{7}x_{15}x_{19}x_{27} x_{35}
x_{39} x_{47} &\stackrel{}{\longmapsto} &  (\beta {\mathcal P}^1
x_{3})^{2}(\beta {\mathcal P}^3{\mathcal P}^1 x_{3})^{2} x_{3}
x_{7}x_{15}x_{19}x_{27} x_{35}
x_{39} x_{47}  \\
\stackrel{}{{\mathcal P}^1}{\uparrow } & &\stackrel{}{{\mathcal P}^1} \uparrow  \\
 {\mathcal P}^1( (\beta {\mathcal P}^1 x_{3})^{2}(\beta {\mathcal P}^3{\mathcal P}^1 x_{3})^{2} x_{3} x_{7}x_{15}x_{19}z_{23} x_{35}
x_{39} x_{47} &\stackrel{}{\longmapsto} & \quad 0  \\
\end{array}
\] }

\noindent Obviously this is a contradiction. So  $\phi_{12}$and
$\phi_{17}$ are not  epimorphisms. This means that $(e_{12})^{*}$,
and $(e_{17})^{*}$
 can not be   split monomorphisms of all Steenrod algebra module. Hence we obtain that

 $$ Mwgt( E_{7}   ; \Bbb{F}_{3} ) \geq  13,  \,\,\,    Mwgt( E_{8}   ; \Bbb{F}_{3} ) \geq  18 \, .$$

\noindent Now we consider the category weight. For $G_{2}$,
$x_{3}x_{5}\in H^{*} (P^{2}(\Omega
  G_2 );\Bbb{F}_{3} )$, so  $(e_{2})^{*}$ is a monomorphism, so  $ wgt( G_{2} ;\Bbb{F}_{3})= 2 $.
  For $F_{4}$,  $(\beta {\mathcal P}^1 x_{3})^{2} x_{3} x_{7} x_{11}x_{15}\in
H^{*} (P^{8}(\Omega
  F_4 );\Bbb{F}_{3} )$, so  $(e_{8})^{*}$ is a monomorphism, so  $ wgt( F_{4} ;\Bbb{F}_{3})= 8 $.
  By the same way,  $ wgt(E_6  ;\Bbb{F}_{3})=10$, $ wgt(E_7  ;\Bbb{F}_{3})=11$, and $ wgt(E_8  ;\Bbb{F}_{3})=16$. Here  the category weight is the same as the
 the filtration length in \cite{T},  that is,
$wgt(G;\Bbb{F}_{3})=f_{3}(G)$.

For the case of $G_{2}$, $F_{4}$,  and $E_{6}$, by dimensional
reason,
  any generator of type  $x_{-}$ can not be of the form
   $Sq^{i}( z_{-} )$ for any $i$ and for any generator of type
   $z$. So  we can not apply the method in (\ref{modu3}). Hence  we do not obtain  any positive difference between the category
   weight and the module category weight.
Hence  we have
$$ \begin{array}{ccl}
Mwgt( G_2 ;\Bbb{F}_{3})&\geq & wgt( G_2 ;\Bbb{F}_{3}) = 2, \qquad
\, Mwgt( F_4 ;\Bbb{F}_{3})  \geq wgt( F_4 ;\Bbb{F}_{3}) = 8,   \\
 Mwgt( E_6 ;\Bbb{F}_{3}) & \geq  & wgt( E_6 ;\Bbb{F}_{3}) = 10 .
 \end{array} $$

\end{proof}

\noindent {\it Remark.} Combined with Toomer's result in \cite{T},
we have the following conclusion:
\bigskip

 \centerline{
 \vbox{\offinterlineskip
\tabskip=0pt \halign {
   \strut   \vrule#& # &  \vrule#& # &
    \vrule#&  #  &  \vrule #&  # &  \vrule#& #\cr
\noalign{\hrule} & \, G  & & \, $wgt(G;\Bbb{F}_{3}) -
cup(G;\Bbb{F}_{3}) $ &&\, $Mwgt(G;\Bbb{F}_{2})-wgt(G;\Bbb{F}_{2})
$&& \, $Mwgt(G;\Bbb{F}_{3})-wgt(G;\Bbb{F}_{3}) $ & \cr
 \noalign{\hrule} &  \, $G_2$ &&
\quad  0 && \quad $\geq$ 0 &&\quad $\geq$ 0 &\cr \noalign{\hrule}
& \, $F_4$ && \quad 2 && \quad $ \geq $ 2 && \quad $ \geq $ 0 &\cr
\noalign{\hrule} & \, $E_{6}$ && \quad 2 &&\quad $ \geq $ 2
&&\quad $ \geq $ 0 &\cr
 \noalign{\hrule} &\,  $E_{7}$ &&
\quad 2 &&\quad $ \geq $ 2  &&\quad $ \geq $ 2 &\cr
\noalign{\hrule} & \, $E_{8}$ && \quad 4 &&\quad  $ \geq $ 0
&&\quad $\geq $ 2 &\cr
 \noalign{\hrule} }   }  }

\end{document}